\documentclass{amsart}
\usepackage[english]{babel}
\usepackage[utf8]{inputenc}
\usepackage{amsmath, amssymb,epic,graphicx,mathrsfs,enumerate}
\usepackage[all]{xy}
\usepackage{color}
\usepackage{comment}
\usepackage{enumitem}
\usepackage{hyperref}
\usepackage[]{frontespizio}

\usepackage{amsmath}
\usepackage{amsthm}
\usepackage{amssymb}
\usepackage{latexsym}
\usepackage{epsfig}

\DeclareMathOperator{\perm}{Sym}

\DeclareMathOperator{\gl}{GL}

\DeclareMathOperator{\End}{End}

\newtheorem{thm}{Theorem}

\numberwithin{equation}{section}

\renewcommand{\footnote}{\endnote}
\newcommand{\ignore}[1]{}\makeglossary

\begin{document}
	\bibliographystyle{amsplain}
\title[2-covering numbers]{2-covering numbers\\ of some finite solvable groups}

\author{Andrea Lucchini}
\address{Andrea Lucchini\\ Universit\`a di Padova\\  Dipartimento di Matematica \lq\lq Tullio Levi-Civita\rq\rq\\ Via Trieste 63, 35121 Padova, Italy\\email: lucchini@math.unipd.it}

\begin{abstract}
A 2-covering for a finite group $G$ is a set
of proper subgroups of $G$ such that every pair of elements of $G$ is contained
in at least one subgroup in the set. The minimal number of subgroups needed
to 2-cover a group $G$ is called the 2-covering number and denoted by $\sigma_2(G).$ In \cite{gk} it is conjectured that if $G$ is solvable and not 2-generated, then $\sigma_2(G)=1+q+q^2,$ where $q$ is a prime power. We disprove this conjecture.
\end{abstract}

\maketitle

\section{Introduction}
A non-cyclic finite group $G$ is the union of its proper subgroups. 
A set of proper subgroups is a covering for  $G$ if its union is the whole
group. The minimal number of subgroups needed to cover $G$ is called the
covering number of $G$ and is denoted by $\sigma(G).$ The behavior of $\sigma(G)$ has been studied in several articles by numerous authors. We recall in particular that Tomkinson \cite{tom} proved that if $G$ is a solvable non-cyclic finite group, then $\sigma(G) = p^k + 1$ for some prime $p.$

In a recent paper, Stephen M. Gagola III and  Joseph Kirtland \cite{gk}
introduced the the concept of a 2-covering for a group $G,$ which is a set of proper subgroups of $G$ such that every pair of elements of $G$ is contained in at least one subgroup in the set.
A finite group $G$ admits a 2-covering if and only if it is not 2-generated. In this case   $\sigma_2(G)$ denotes
the smallest size of a 2-covering of $G.$ 

In \cite{gk} the authors determine the 2-covering number for finite nilpotent groups and particular
classes of finite solvable groups. They prove that $\sigma_2(G)=1+p+p^2$ if $G$ is a finite $p$-group with $d(G)>2,$ and propose the following conjecture (which is true for example if $G$ is supersolvable):
{\sl{if $G$ is a finite solvable group, with $d(G)\geq 3,$ then
$\sigma_2(G)=p^{2t}+p^t+1,$ for some prime divisor $p$ of the order of $G$ and some positive integer $t$.}} In this short note we disprove this conjecture. Note that $p^{2t}+p^t+1$ is an odd integer for every choice of $p$ and $t$. On the contrary, we will prove that there exist infinitely many even integers that are the 2-covering number of some finite solvable group. In particular the following two results hold.

\begin{thm}\label{uno}For every odd prime $p$, there exists a finite solvable group $G$ with $\sigma_2(G)=1+p^2+p^3+p^4.$
\end{thm}
\begin{thm}\label{due}
	Let $q$ be a prime power and $p$ an odd  prime dividing $q+1$. Then there exists a finite solvable group $G$ with $\sigma_2(G)=q^2+q^3+q^4+p.$
\end{thm}

\section{Proofs of Theorems \ref{uno} and \ref{due}}
	Let $H$ be a finite solvable group with $d(H)=2$ and let $V$ be a faithful irreducible $H$-module.
	Moreover, let $\mathbb{F}=\End_H(V),$ $q=|\mathbb{F}|$, $r=\dim_{\mathbb{F}}V$ and consider the semidirect product $G=V^{r+1}\rtimes H.$ 
	It follows from \cite[Satz 4]{ga} that $d(V^t\rtimes H)\leq n$ if and only if $t\leq r(n-1).$ In particular, $d(G)=3.$ The maximal subgroups of $G$ are of two possible types:
	\begin{enumerate}
		\item $M=WH^v,$ where $W$ is a maximal $H$-submodule of $V^{r+1}$ and $v\in V^{r+1}.$
		\item $M=V^{r+1}K$ with $K$ a maximal subgroup of $H.$
	\end{enumerate}
	Notice that the number of maximal subgroups of $G$ of the first type is precisely $$\gamma=q^r\left( \frac{q^{r+1}-1}{q-1}\right)=\frac{q^{2r+1}-q^r}{q-1}.$$ Indeed, the number of maximal $H$-submodules $W$ of the direct power $V^{r+1}$ is $(q^{r+1}-1)/(q-1),$ and, for a given choice of $W$, $WH^{v_1}=WH^{v_2}$ if and only if $v_1-v_2\in W.$

	Now let $\mathcal M$ be a $2$-covering of $G$. It is not restrictive to assume that every element of $\mathcal M$ is a maximal subgroup of $G$.
	If $M$ is a maximal subgroup of $G$ of the first type, then, again by \cite[Satz 4]{ga},
	$d(M)=2$ and this implies that $M$ must belong to $\mathcal{M}.$ To determine the number of maximal subgroups of the second type that $\mathcal M $ must contain, it is necessary to specify in detail the action of $H$ on $V.$ But before doing this, let us add one more piece of general information. For every positive integer $u$, denote by $d_H(V^u)$ the smallest number of elements needed to generate $V^u$ as an $H$-module. Let  $\mathbb K$ be the prime subfield of $\mathbb F.$ Note that
	$d_H(V^u)\leq d$ if and only if there exists a $\mathbb KH$-epimorphism $\phi$ from $(\mathbb KH)^d$ to $V^u$. 
	Since $V^u$ is a completely reducible $\mathbb KH$-module, $\ker \phi$ contains $({\rm J}(\mathbb KH))^d,$  where  ${\rm J}(\mathbb KH)$ denotes the Jacobson radical of $ \mathbb KH.$ Hence $d_H(V^u)\leq d$ if and only if $V^u$ is an epimorphic image of $(\mathbb KH/{\rm J}(\mathbb KH))^d.$ Since the multiplicity of $V$ in $\mathbb KH/{\rm J}(\mathbb KH)$ is precisely $r=\dim_{\mathbb F}V$, we conclude that  $d_H(V^u)\leq d$ if and only if $rd\geq u.$ 
	In particular
	\begin{equation}\label{genmod} d_H(V^u)=1 \text { if } u\leq r \text { and } d_H(V^{r+1})=2.
	\end{equation}
	
	\begin{proof}[Proof of Theorem \ref{uno}]
	Let $H$ be the quaternion group of order 8. For every odd prime $p$, $H$ can be identified with an absolutely irreducible of $\gl(2,p).$ Hence there exists an irreducible $H$-module $V$ of order $p^2$, and $\mathbb F=\End_H(V)$ is the field with $p$ elements. In particular $r=\dim_{\mathbb F}V=2$ and we may consider $G=V^3\rtimes H.$ Let $\mathcal M$ be a 2-covering of $G.$ As we notice before, $\mathcal M$ must contain all the maximal subgroups of $G$ of the first type. Hence
	$\sigma_2(G)\geq \gamma=p^2+p^3+p^4.$
	
Let now $X$ be a 2-generated subgroup of $G$. Then $X=\langle h_1v_1, h_2v_2\rangle$ with $h_1, h_2\in H$ and $v_1,v_2 \in V^3$. We consider the different possibilities for $h_1$ and $h_2.$ 
	
	\noindent a) $\langle h_1, h_2 \rangle =H.$  In this case no maximal subgroup of the second type can contain $X$, and therefore $X\leq M$ for some maximal subgroup of $G$ of the first type.
	
	\noindent b) $\langle h_1, h_2\rangle$ is cyclic of order $m\in \{2,4\}.$
	In this case we may assume $X=\langle hv_1, v_2\rangle$ with $|h|=m$. Since the action of $H$ on $V$ is fixed-point-free, $C_{V^3}(h)=0$ and therefore $hv_1$ is conjugate to $h$ in $G$, Hence  we may assume $X=\langle h, v\rangle^g$ for some $v\in V.$ Moreover it follows from (\ref{genmod}) that there is a maximal $H$-submodule $W$ of $V^{r+1}$ containing $v$ and therefore $X\leq WH^g.$
	
	\noindent c) $X\leq V^3.$ By (\ref{genmod}), there exist
	$v_1, v_2$ generating $V^3$ as an $H$-module. This implies that $\langle v_1, v_2, H^g\rangle=G$ for every $g\in G$ and therefore no maximal subgroup of $G$ of the first type can contain $X$. Hence $\mathcal M$ must contain a maximal subgroup of the second type, and clearly  any such subgroup contains all  the 2-generated subgroups of $V^3.$

	From the previous analysis, it follows that $\sigma_2(G)=\gamma+1=1+p^2+p^3+p^4.$
		\end{proof}
		
		\begin{proof}[Proof of Theorem \ref{due}]
		Let $q$ be a prime power and $p$ a prime which divides $q+1.$ Consider a dihedral group $H$ of order $2p$ and let $\mathbb E$ be the field with $q^2$ elements.
	Let $V$ be the additive group of $\mathbb E$ and $a, b \in H$ of order  $p$ and 2, respectively. We may identify $\langle a\rangle $ with a subgroup of the multiplicative group $\mathbb E^\times$ and we may define a faithful irreducible action of $H$ on $V$ by setting $v^a=va$ and $v^b=v^{q}$ for every $v\in V.$ It turns out that $V$ is $H$-irreducible and that $\End_H(V)$ is the field with $q$ elements. As a consequence, $r=2$. Let $\mathcal M$ be a 2-covering of $G.$ As in the previous proof, $\mathcal M$ must contain all the maximal subgroups of $G$ of the first type and therefore
	$\sigma_2(G)\geq \gamma=q^2+q^3+q^4.$
	
	Let now $X=\langle h_1v_1, h_2v_2\rangle$ be a 2-generated subgroup of $G$, with $h_1, h_2\in H$ and $v_1,v_2 \in V^3$. If $h\in H$ has order $p,$ then $C_V(h)=0.$ Hence,
	with the same arguments as in the previous proof, we may conclude that if $p$ divides the order of $\langle h_1, h_2\rangle$, then $X$ is contained in some maximal subgroup of $G$ of the first type. However, let $h$ be an element of $H$ of order 2. Then there exists $0\neq u \in C_V(h)$.  Moreover, by (\ref{genmod}), there exists $(w_1,w_2)\in V^2$ which generates $V^2$ as an $H$-module.
	Now let $v_1=(u,0,0)$ and $v_2=(0,w_1,w_2)$. Then
	$X=\langle hv_1, v_2\rangle= \langle h, v_1, v_2\rangle$ cannot be contained in a maximal subgroup of $G$ of the first type, since the normal closure of $\langle v_1, v_2\rangle$ in $G$ coincides with $V^3.$
 Hence the unique maximal subgroup of $G$ which contains $X$ is $V^{3}\langle h \rangle.$ This implies that $\mathcal M$ must contain all the  maximal subgroups of $G$ of the form $V^3\langle h \rangle,$ with $|h|=2.$ 
We conclude that $\sigma_2(G)=\gamma+p=q^2+q^3+q^4+p.$
\end{proof}

\end{document}